\documentclass[11pt]{article}
\usepackage[top=1.25in, bottom=1.5in, left=1.25in, right=1.25in]{geometry}
\usepackage{amsfonts}
\usepackage{amsthm}
\usepackage{amsmath}
\usepackage{amssymb}
\usepackage{tikz}
\usepackage{array}
\usepackage{url}
\usepackage{authblk}

\newcolumntype{x}[1]{>{\centering\arraybackslash\hspace{0pt}}p{#1}}

\newtheorem{theorem}{Theorem}[section]
\newtheorem{definition}{Definition}[section]
\newtheorem{example}{Example}[section]
\newtheorem{remark}{Remark}[section]

\newtheorem{construction}{Construction}

\newcommand{\BB}{{\ensuremath{\mathcal{B}}}}
\newcommand{\zed}{{\ensuremath{\mathbb{Z}}}}
\newcommand{\GG}{{\ensuremath{\mathcal{G}}}}
\newcommand{\AAA}{{\ensuremath{\mathcal{A}}}}

\begin{document}

%\nocite{*}

\title{In the Frame\thanks{The author's research is supported by  NSERC discovery grant RGPIN-03882.}}
\author[1,2]{Douglas R.\ Stinson}
\affil[1]{David R.\ Cheriton School of Computer Science\\ University of Waterloo\\
Waterloo, Ontario, N2L 3G1, Canada
%\\
%{\tt dstinson@uwaterloo.ca}
}
\affil[2]{School of Mathematics and Statistics\\
Carleton University\\
Ottawa, Ontario, K1S 5B6, Canada}

\date{
\today
}

\maketitle

\begin{abstract}
In this expository paper, I survey Room frames and Kirkman frames, concentrating on the early history of these objects. I mainly look at the basic construction techniques, but I also provide some historical remarks and discussion. I also briefly discuss some other types of frames that have been investigated as well as some applications of frames to the construction of other types of designs.

\end{abstract}

\section{Room Squares and Room Frames}
\label{intro.sec}

I begin by briefly discussing the Room square problem. I refer to the 1992 survey by Dinitz and Stinson \cite{DS92a} for a thorough summary of the history of Room squares and related designs up to that time. 

\begin{definition}
A \emph{Room square of side $n$} is an $n$ by $n$ array, $F$, on a set $S$ of $n+1$ symbols, that satisfies the following properties:
\begin{enumerate}
\item every cell of $F$ either is empty or contains an unordered pair of symbols from $S$,
\item each symbol in $S$ occurs in exactly one cell in each row and each column of $F$, and 
\item every unordered pair of symbols occurs in exactly one cell of $F$.
\end{enumerate}
\end{definition}
A Room square of side seven is presented in Figure \ref{RS-7}.

It is clear that $n$ must be an odd positive integer if a Room square of side $n$ exists. Although Room squares have been studied since the 1850's, it was not until 1974 that a complete existence result was given by Wallis \cite{Wa74}, as a culmination of work by Mullin, Nemeth, Wallis and others. See \cite{MW75} for a short self-contained proof of the existence of Room squares.

\begin{theorem} There exists a Room square of side $n$ if and only if $n$ is an odd positive integer and $n \neq 3,5$.
\end{theorem}

\begin{figure}

\begin{center}

%\vspace*{.2in}

\begin{tikzpicture}[scale=0.15]

\draw [thick] (0,0) -- (0,35);
\draw [thick] (5,0) -- (5,35);
\draw [thick] (10,0) -- (10,35);
\draw [thick] (15,0) -- (15,35);
\draw [thick] (20,0) -- (20,35);
\draw [thick] (25,0) -- (25,35);
\draw [thick] (30,0) -- (30,35);
\draw [thick] (35,0) -- (35,35);

\draw [thick] (0,0) -- (35,0);
\draw [thick] (0,5) -- (35,5);
\draw [thick] (0,10) -- (35,10);
\draw [thick] (0,15) -- (35,15);
\draw [thick] (0,20) -- (35,20);
\draw [thick] (0,25) -- (35,25);
\draw [thick] (0,30) -- (35,30);
\draw [thick] (0,35) -- (35,35);

\node at (12.5,2) {$0,4$};
\node at (22.5,2) {$3,5$};
\node at (27.5,2) {$1,2$};
\node at (32.5,2) {$7,6$};

\node at (7.5,7) {$6,3$};
\node at (17.5,7) {$2,4$};
\node at (22.5,7) {$0,1$};
\node at (27.5,7) {$7,5$};

\node at (2.5,12) {$5,2$};
\node at (12.5,12) {$1,3$};
\node at (17.5,12) {$6,0$};
\node at (22.5,12) {$7,4$};

\node at (7.5,17) {$0,2$};
\node at (12.5,17) {$5,6$};
\node at (17.5,17) {$7,3$};
\node at (32.5,17) {$4,1$};

\node at (2.5,22) {$6,1$};
\node at (7.5,22) {$4,5$};
\node at (12.5,22) {$7,2$};
\node at (27.5,22) {$3,0$};

\node at (2.5,27) {$3,4$};
\node at (7.5,27) {$7,1$};
\node at (22.5,27) {$2,6$};
\node at (32.5,27) {$5,0$};

\node at (2.5,32) {$7,0$};
\node at (17.5,32) {$1,5$};
\node at (27.5,32) {$4,6$};
\node at (32.5,32) {$2,3$};

\end{tikzpicture}

%\vspace{.2in}

\end{center}
\caption{A Room square of side seven}
\label{RS-7}
\end{figure}

\begin{figure}[t]

\begin{center}

%\vspace*{.2in}

\begin{tikzpicture}[scale=0.15]

\node at (5,45) {\Large{$0,5$}};
\node at (15,35) {\Large{$1,6$}};
\node at (25,25) {\Large{$2,7$}};
\node at (35,15) {\Large{$3,8$}};
\node at (45,5) {\Large{$4,9$}};

\draw [thick] (0,0) -- (0,50);
\draw [thick] (5,0) -- (5,40);
\draw [thick] (10,0) -- (10,50);
\draw [thick] (15,0) -- (15,30);
\draw [thick] (15,40) -- (15,50);
\draw [thick] (20,0) -- (20,50);
\draw [thick] (25,0) -- (25,20);
\draw [thick] (25,30) -- (25,50);
\draw [thick] (30,0) -- (30,50);
\draw [thick] (35,0) -- (35,10);
\draw [thick] (35,20) -- (35,50);
\draw [thick] (40,0) -- (40,50);
\draw [thick] (45,10) -- (45,50);
\draw [thick] (50,0) -- (50,50);

\draw [thick] (0,0) -- (50,0);
\draw [thick] (0,5) -- (40,5);
\draw [thick] (0,10) -- (50,10);
\draw [thick] (0,15) -- (30,15);
\draw [thick] (40,15) -- (50,15);
\draw [thick] (0,20) -- (50,20);
\draw [thick] (0,25) -- (20,25);
\draw [thick] (30,25) -- (50,25);
\draw [thick] (0,30) -- (50,30);
\draw [thick] (0,35) -- (10,35);
\draw [thick] (20,35) -- (50,35);
\draw [thick] (0,40) -- (50,40);
\draw [thick] (10,45) -- (50,45);
\draw [thick] (0,50) -- (50,50);

\node at (7.5,2) {$3,6$};
\node at (17.5,2) {$2,8$};
\node at (27.5,2) {$1,5$};
\node at (37.5,2) {$0,7$};

\node at (2.5,7) {$1,3$};
\node at (12.5,7) {$7,8$};
\node at (22.5,7) {$0,6$};
\node at (32.5,7) {$2,5$};

\node at (7.5,12) {$1,7$};
\node at (17.5,12) {$4,5$};
\node at (27.5,12) {$6,9$};
\node at (42.5,12) {$0,2$};

\node at (2.5,17) {$2,6$};
\node at (12.5,17) {$0,9$};
\node at (22.5,17) {$1,4$};
\node at (47.5,17) {$5,7$};

\node at (2.5,22) {$8,9$};
\node at (12.5,22) {$3,5$};
\node at (37.5,22) {$4,6$};
\node at (47.5,22) {$0,1$};

\node at (7.5,27) {$4,8$};
\node at (17.5,27) {$0,3$};
\node at (32.5,27) {$1,9$};
\node at (42.5,27) {$5,6$};

\node at (7.5,32) {$2,9$};
\node at (22.5,32) {$5,8$};
\node at (32.5,32) {$0,4$};
\node at (42.5,32) {$3,7$};

\node at (2.5,37) {$4,7$};
\node at (27.5,37) {$0,8$};
\node at (37.5,37) {$5,9$};
\node at (47.5,37) {$2,3$};

\node at (17.5,42) {$7,9$};
\node at (27.5,42) {$3,4$};
\node at (37.5,42) {$1,2$};
\node at (47.5,42) {$6,8$};

\node at (12.5,47) {$2,4$};
\node at (22.5,47) {$3,9$};
\node at (32.5,47) {$6,7$};
\node at (42.5,47) {$1,8$};

%\node at (45,5) {$4,9$};
%\node at (35,15) {$3,8$};
%\node at (25,25) {$2,7$};
%\node at (15,35) {$1,6$};
%\node at (5,45) {$0,5$};

\end{tikzpicture}

%\vspace{.2in}

\end{center}
\caption{A Room frame of type $2^5$}
\label{type2^5}
\end{figure}

\begin{figure}[t]

\begin{center}

\begin{tikzpicture}[scale=0.15]

\draw [thick] (0,0) -- (0,60);
\draw [thick] (5,0) -- (5,50);
\draw [thick] (10,0) -- (10,60);
\draw [thick] (15,0) -- (15,40);
\draw [thick] (15,50) -- (15,60);
\draw [thick] (20,0) -- (20,60);
\draw [thick] (25,0) -- (25,30);
\draw [thick] (25,40) -- (25,60);
\draw [thick] (30,0) -- (30,60);
\draw [thick] (35,0) -- (35,20);
\draw [thick] (35,30) -- (35,60);
\draw [thick] (40,0) -- (40,60);
\draw [thick] (45,0) -- (45,10);
\draw [thick] (45,20) -- (45,60);
\draw [thick] (50,0) -- (50,60);
\draw [thick] (55,10) -- (55,60);
\draw [thick] (60,0) -- (60,60);

\draw [thick] (0,0) -- (60,0);
\draw [thick] (0,5) -- (50,5);
\draw [thick] (0,10) -- (60,10);
\draw [thick] (0,15) -- (40,15);
\draw [thick] (50,15) -- (60,15);
\draw [thick] (0,20) -- (60,20);
\draw [thick] (0,25) -- (30,25);
\draw [thick] (40,25) -- (60,25);
\draw [thick] (0,30) -- (60,30);
\draw [thick] (0,35) -- (20,35);
\draw [thick] (30,35) -- (60,35);
\draw [thick] (0,45) -- (10,45);
\draw [thick] (20,45) -- (60,45);
\draw [thick] (0,40) -- (60,40);
\draw [thick] (0,50) -- (60,50);
\draw [thick] (10,55) -- (60,55);
\draw [thick] (0,60) -- (60,60);

\node at (7.5,2) {$4,7$};
\node at (17.5,2) {$0,8$};
\node at (27.5,2) {$1,9$};
\node at (37.5,2) {$2,5$};
\node at (47.5,2) {$3,6$};

\node at (2.5,7) {$2,9$};
\node at (12.5,7) {$3,5$};
\node at (22.5,7) {$4,6$};
\node at (32.5,7) {$0,7$};
\node at (42.5,7) {$1,8$};

\node at (7.5,12) {$1,3$};
\node at (17.5,12) {$7,\mathrm{B}$};
\node at (27.5,12) {$5,6$};
\node at (37.5,12) {$0,\mathrm{A}$};
\node at (57.5,12) {$2,8$};

\node at (2.5,17) {$6,8$};
\node at (12.5,17) {$2,\mathrm{B}$};
\node at (22.5,17) {$0,1$};
\node at (32.5,17) {$5,\mathrm{A}$};
\node at (52.5,17) {$3,7$};

\node at (7.5,22) {$6,\mathrm{B}$};
\node at (17.5,22) {$5,9$};
\node at (27.5,22) {$4,\mathrm{A}$};
\node at (47.5,22) {$0,2$};
\node at (57.5,22) {$1,7$};

\node at (2.5,27) {$1,\mathrm{B}$};
\node at (12.5,27) {$0,4$};
\node at (22.5,27) {$9,\mathrm{A}$};
\node at (42.5,27) {$5,7$};
\node at (52.5,27) {$2,6$};

\node at (7.5,32) {$8,9$};
\node at (17.5,32) {$3,\mathrm{A}$};
\node at (37.5,32) {$1,4$};
\node at (47.5,32) {$5,\mathrm{B}$};
\node at (57.5,32) {$0,6$};

\node at (2.5,37) {$3,4$};
\node at (12.5,37) {$8,\mathrm{A}$};
\node at (32.5,37) {$6,9$};
\node at (42.5,37) {$0,\mathrm{B}$};
\node at (52.5,37) {$1,5$};

\node at (7.5,42) {$2,\mathrm{A}$};
\node at (27.5,42) {$0,3$};
\node at (37.5,42) {$9,\mathrm{B}$};
\node at (47.5,42) {$7,8$};
\node at (57.5,42) {$4,5$};

\node at (2.5,47) {$7,\mathrm{A}$};
\node at (22.5,47) {$5,8$};
\node at (32.5,47) {$4,\mathrm{B}$};
\node at (42.5,47) {$2,3$};
\node at (52.5,47) {$0,9$};

\node at (17.5,52) {$2,4$};
\node at (27.5,52) {$8,\mathrm{B}$};
\node at (37.5,52) {$6,7$};
\node at (47.5,52) {$1,\mathrm{A}$};
\node at (57.5,52) {$3,9$};

\node at (12.5,57) {$7,9$};
\node at (22.5,57) {$3,\mathrm{B}$};
\node at (32.5,57) {$1,2$};
\node at (42.5,57) {$6,\mathrm{A}$};
\node at (52.5,57) {$4,8$};

\node at (5,55) {\Large{$0,5$}};
\node at (15,45) {\Large{$1,6$}};
\node at (25,35) {\Large{$2,7$}};
\node at (35,25) {\Large{$3,8$}};
\node at (45,15) {\Large{$4,9$}};
\node at (55,5) {\Large{$\mathrm{A},\mathrm{B}$}};

\end{tikzpicture}

\end{center}
\caption{A Room frame of type $2^6$}
\label{type2^6}
\end{figure}

In this paper, I am interested in a generalization of a Room square called a Room frame. Here is a definition. 
\begin{definition}
Let $t$ and $u$ be positive integers and let  $S$ be a set of $tu$ symbols. Suppose that $S$ is partitioned into $u$ sets of size $t$, denoted $S_i$, $1 \leq i \leq u$. A \emph{Room frame of type $t^u$} is a $tu$ by $tu$ array, $F$,  that satisfies the following properties:
\begin{enumerate}
\item the rows and columns of $F$ are indexed by $S$,
\item every cell of $F$ either is empty or contains an unordered pair of symbols from $S$,
\item the subarrays of $F$ indexed by $S_i \times S_i$ are empty, for $1 \leq i \leq u$ (these empty subarrays are called \emph{holes}),
\item if $s \in S_i$, then row or column $s$ contains each symbol in $S \setminus S_i$ exactly once, and 
\item the unordered pairs of symbols occurring in $F$ are precisely the pairs
$\{x,y\}$, where $x, y$ are in different $S_i$'s (each such pair occurs in one cell of $F$).
\end{enumerate}
\end{definition}
Room frames of types $2^5$ and $2^6$ are depicted in Figures \ref{type2^5} and \ref{type2^6}, respectively.

The Room frames defined above are \emph{uniform}, which means that all the holes have the same size. Non-uniform Room frames have also received much study, but I mainly focus on the uniform case in this paper.

One initial observation is that a Room square of side $n$ equivalent to a Room frame of type $1^n$. Suppose $F$ is a Room square of side $n$. Pick a particular symbol $x$ and permute the rows and columns of $F$ so the cells containing $x$ are precisely the cells on the main diagonal (such a Room square is said to be \emph{standardized}). Then delete the pairs in these cells. Conversely, given a Room frame of type $1^n$, then we can introduce a new a symbol $x$ and place the pair $\{x,s\}$ in the cell $(s,s)$ for all $s$.

For example, suppose we start with the Room square of side seven that was presented in Figure \ref{RS-7}.
This Room square is already standardized, so we simply remove the pairs in the diagonal cells to construct a Room frame of type $1^7$. This Room frame is presented in Figure \ref{type1^7}. 

\begin{figure}

\begin{center}

%\vspace*{.2in}

\begin{tikzpicture}[scale=0.2]

\draw [thick] (0,0) -- (0,35);
\draw [thick] (5,0) -- (5,35);
\draw [thick] (10,0) -- (10,35);
\draw [thick] (15,0) -- (15,35);
\draw [thick] (20,0) -- (20,35);
\draw [thick] (25,0) -- (25,35);
\draw [thick] (30,0) -- (30,35);
\draw [thick] (35,0) -- (35,35);

\draw [thick] (0,0) -- (35,0);
\draw [thick] (0,5) -- (35,5);
\draw [thick] (0,10) -- (35,10);
\draw [thick] (0,15) -- (35,15);
\draw [thick] (0,20) -- (35,20);
\draw [thick] (0,25) -- (35,25);
\draw [thick] (0,30) -- (35,30);
\draw [thick] (0,35) -- (35,35);

\node at (12.5,2) {$0,4$};
\node at (22.5,2) {$3,5$};
\node at (27.5,2) {$1,2$};

\node at (7.5,7) {$6,3$};
\node at (17.5,7) {$2,4$};
\node at (22.5,7) {$0,1$};

\node at (2.5,12) {$5,2$};
\node at (12.5,12) {$1,3$};
\node at (17.5,12) {$6,0$};

\node at (7.5,17) {$0,2$};
\node at (12.5,17) {$5,6$};
\node at (32.5,17) {$4,1$};

\node at (2.5,22) {$6,1$};
\node at (7.5,22) {$4,5$};
\node at (27.5,22) {$3,0$};

\node at (2.5,27) {$3,4$};
\node at (22.5,27) {$2,6$};
\node at (32.5,27) {$5,0$};

\node at (17.5,32) {$1,5$};
\node at (27.5,32) {$4,6$};
\node at (32.5,32) {$2,3$};

\node at (2.5,32.5) {\Large{$0$}};
\node at (7.5,27.5) {\Large{$1$}};
\node at (12.5,22.5) {\Large{$2$}};
\node at (17.5,17.5) {\Large{$3$}};
\node at (22.5,12.5) {\Large{$4$}};
\node at (27.5,7.5) {\Large{$5$}};
\node at (32.5,2.5) {\Large{$6$}};

\end{tikzpicture}

%\vspace{.2in}

\end{center}
\caption{A Room frame of type $1^7$}
\label{type1^7}
\end{figure}

Perhaps surprisingly, there is no Room frame of type $2^4$ (this was shown by exhaustive case analysis in \cite{St81a}). There is also 
no Room frame of type $1^5$, since such a structure would be equivalent to a (nonexistent) Room square of side five. It is also clear that $u \geq  4$ and  $t(u - 1)$ is even if a Room frame of type $t^u$ exists.  The following theorem gives complete existence results for uniform Room frames. This result is the culmination of many papers spanning a time period of $30$ years.

\begin{theorem}
\label{uniformRoomframe}
\cite{DS81,DL93,DSZ94,DW10,GZ93}
There exists a Room frame of type $t^u$ if and only if $u \geq  4$, $t(u - 1)$ is even,
and $(t, u) \neq (1, 5)$ or $(2, 4)$.
\end{theorem}

I now define a special type of Room frame that has received considerable attention.

\begin{definition}
\label{skew.defn}
Suppose $S$ is a symbol set that is partitioned into $u$ sets of size $t$, denoted $S_i$, $1 \leq i \leq u$. A Room frame $F$ of type $t^u$ is \emph{skew} if, for any $r$ and $s$ where $r$ and $s$ are not in the same $S_i$, precisely one of the two cells $F(r,s)$ and $F(s,r)$ is empty. This skew definition also applies to Room squares. 
\end{definition}

Skew Room frames are more difficult to construct than ``ordinary'' (i.e., non-skew) Room frames. 
In the case of skew Room squares, the following result was shown in 1981.

\begin{theorem} 
\cite{St81c}
\label{srs.thm}
There exists a skew Room square of side $n$ if and only if $n$ is an odd positive integer and $n \neq 3,5$.
\end{theorem}

Theorem \ref{srs.thm} was the consequence of a long series of papers by a number of different authors.  
The paper \cite{St81c} gives a short proof of Theorem \ref{srs.thm} that is based on skew Room frames.

For skew Room frames of type $t^u$, the following theorem is the current state of knowledge. Note that this theorem
generalizes Theorem \ref{srs.thm}, which is the special case $t=1$.

\begin{theorem}
\cite{St87a,CZ96,ZG07}
There exists a skew Room frame of type $t^u$ if and only if $u \geq  4$, $t(u - 1)$ is even,
$(t, u) \neq (1, 5)$ or $(2, 4)$, with the following possible exceptions:
\begin{enumerate}
\item $u = 4$ and $t \equiv 2 \pmod{4}$;
\item $u = 5$ and $t \in \{17, 19, 23, 29, 31\}$.
\end{enumerate}
\end{theorem}

\section{Constructions for Room Frames}

\subsection{Orthogonal Starters}

Constructions for Room frames come in two flavours: direct and recursive.
The main direct construction is based on orthogonal frame starters. (It is also possible to construct ``random-looking'' examples of Room squares and Room frames using hill-climbing algorithms; see \cite{DS87,DS92a}.) I briefly discuss the method of orthogonal (frame) starters in this section. For more information on frame starters and orthogonal frame starters, see the recent survey \cite{St22}. Note that all of our definitions will refer to additive abelian groups. 

\begin{definition}
Let $G$ be an abelian group of order $g$ and let $H$ be a subgroup of order $h$. A \emph{frame starter} in $G \setminus H$ is a set of $(g-h)/2$ pairs $\{ \{x_i,y_i \} : 1 \leq i \leq (g-h)/2\}$ that satisfies the following two properties:
\begin{enumerate}
\item $\{ x_i, y_i :  1 \leq i \leq (g-h)/2 \} = G \setminus H$.
\item $\{ \pm (x_i-y_i) :  1 \leq i \leq (g-h)/2 \} = G \setminus H$.
\end{enumerate}
\end{definition}
In other words, the pairs in the frame starter form a partition of $G \setminus H$, and the differences obtained from these pairs also partitions $G \setminus H$.

\begin{definition}
Suppose that $S_1 = \{ \{x_i,y_i \} : 1 \leq i \leq (g-h)/2\}$ and $S_2 = \{ \{u_i,v_i \} : 1 \leq i \leq (g-h)/2\}$
are both frame starters in $G \setminus H$. Without loss of generality, assume that
$y_i - x_i = v_i - u_i$ for $1 \leq i \leq (g-h)/2$.  We say that $S_1$ and $S_2$ are
\emph{orthogonal} if the following two properties hold:
\begin{enumerate}
\item $y_i - v_i \not\in H$ for $1 \leq i \leq (g-h)/2$.
\item $y_i - v_i \neq y_j - v_j$ if $1 \leq i,j \leq (g-h)/2$, $i \neq j$. 
\end{enumerate}
\end{definition}
Thus, when we match up the pairs in $S_1$ and $S_2$ according to their differences, the ``translates'' are distinct elements of $G \setminus H$. These translates are often called an \emph{adder}.

\begin{example}
Suppose $G =\zed_{10}$ and $H = \{0,5\}$. Here are two orthogonal frame starters in $G\setminus H$:
\[
\begin{array}{l}
S_1 = \{ \{3,4\}, \{7,9\}, \{8,1\}, \{2,6\} \} \\
S_2 = \{ \{6,7\}, \{1,3\}, \{9,2\}, \{4,8\} \}. 
\end{array}
\]
\end{example}

\begin{example}
Suppose $G =\zed_{7}$ and $H = \{0\}$. Here are two orthogonal frame starters in $G\setminus H$:
\[
\begin{array}{l}
S_1 = \{ \{3,4\}, \{2,5\}, \{1,6\} \} \\
S_2 = \{ \{2,3\}, \{5,1\}, \{6,4\} \}. 
\end{array}
\]
\end{example}

\begin{example}
\label{z15.exam}
Suppose $G =\zed_{15}$ and $H = \{0,5,10\}$. Here are two orthogonal frame starters in $G\setminus H$:
\[
\begin{array}{l}
S_1 = \{ \{1,2\}, \{9,11\}, \{3,6\}, \{8,12\}, \{13,4\}, \{7,14\} \} \\
S_2 = \{ \{2,3\}, \{11,13\}, \{9,12\} , \{4,8\}, \{1,7\}, \{14,6\}\}. 
\end{array}
\]
\end{example}

\begin{example}
\label{z4z4.exam}
Suppose $G =\zed_{4} \times \zed_{4}$ and $H = \{(0,0), (0,2), (2,0), (2,2)\}$. Here are two orthogonal frame starters in $G\setminus H$:
\[
\begin{array}{ll}
S_1 = &\{ \{(1,1),(3,2)\}, \{(3,0),(3,1)\}, \{(2,1),(3,3)\}, \{(0,3),(1,3)\},\\
&   \{(1,0),(2,3)\}, \{(0,1),(1,2)\} \} \\
S_2 = &\{ \{(1,2),(3,3)\}, \{(1,3),(1,0)\}, \{(1,1),(2,3)\}, \{(3,1),(0,1)\},\\
&   \{(2,1),(3,0)\}, \{(3,2),(0,3)\} \} \\
\end{array}
\]
\end{example}

Orthogonal frame starters can be used to construct a Room frame of the relevant type. The resulting Room frame has $G$ in its automorphism group. In the case where $|H| = 1$, then we have a \emph{starter}. Orthogonal starters can be used to generate Room squares.

Suppose that $S_1 = \{ \{x_i,y_i \} : 1 \leq i \leq (g-h)/2\}$ and $S_2 = \{ \{u_i,v_i \} : 1 \leq i \leq (g-h)/2\}$
are orthogonal frame starters in $G \setminus H$, where $y_i - x_i = v_i - u_i$ for $1 \leq i \leq (g-h)/2$. 
$S_1$ and $S_2$ are \emph{skew-orthogonal} if $y_i - v_i \neq -(y_j - v_j)$ if $1 \leq i,j \leq (g-h)/2$, $i \neq j$. Equivalently, the set of adders and their negatives is precisely $G \setminus H$. It can be verified that all four examples of orthogonal frame starters presented above are in fact skew-orthogonal. As one would suspect, skew-orthogonal starters give rise to skew Room frames (e.g., see Figure \ref{type1^7}).

It has been proven that orthogonal frame starters cannot be used to construct a Room frame of type $2^6$ (in fact, there is no frame starter in $G \setminus H$ when $|G| = 12$ and $|H| = 2$). However, a modification known as \emph{orthogonal intransitive frame starters} permit the construction of these (and many other useful) Room frames. The Room frame depicted in Figure \ref{type2^6} illustrates the basic idea. The ten by ten square in the upper left of the diagram is developed modulo $10$, similar to a Room frame obtained from orthogonal frame starters, except that there are two additional fixed points, denoted $A$ and $B$. The last two rows and the last two columns are developed modulo $10$, and there is a hole containing the two fixed points. We do not give a formal definition, but the associated orthogonal intransitive frame starters, denoted by the quadruple $(S_1,C,S_2,R)$, are defined on 
$\zed_{10} \cup \{A,B\}$ and they are obtained from the first row and the first column of the Room frame. Note that $R$ and $C$ refer  to the last two rows and the last two columns of the Room frame, respectively. 
\[
\begin{array}{l}
S_1 = \{ \{7,9\}, \{1,2\}, \{6,A\}, \{3,B\} \}  \\
C =  \{ \{4,8\} \}\\
S_2 = \{ \{6,8\}, \{3,4\}, \{7,A\}, \{1,B\} \}\\ 
R = \{ \{2,9\} \}.
\end{array}
\]

\subsection{Existence of Uniform Room Frames}

I now discuss some aspects of the proof of Theorem \ref{uniformRoomframe}.
The cases $t=1$ are of course equivalent to Room squares. Thus there exists a Room frame of type $1^u$ if and only if $u$ is odd and $u \geq 7$. The most important cases in establishing the general existence result (Theorem \ref{uniformRoomframe}) are when $t=2$ or $t=4$. I examine these cases now.

\subsubsection{The Case $t=2$} %\quad  \vspace{-.1in}\\

Since a Room frame of type $2^4$ does not exist, the goal was to prove that there is a Room frame of type $2^u$ for all $u \geq 5$. Jeff Dinitz and I considered this problem in detail in \cite{DS81}. 

Our approach used pairwise balanced designs (PBDs). A \emph{pairwise balanced design} is a pair
$(X,\AAA)$, where $X$ is a finite set of \emph{points} and $\AAA$ is a set of subsets of $X$ (called \emph{blocks}), with the property that every pair of points is contained in a unique block. Assume that $|A| > 1$ for every $A \in \AAA$. We say that $(X,\AAA)$ is a $(v,K)$-PBD if $|X| = v$ and $|A| \in K$ for every $A \in \AAA$. Also, a set $K$ is \emph{PBD-closed} if $v \in K$ whenever there exists a $(v,K)$-PBD. Finally, if $K = \{k\}$, then a $(v,K)$-PBD is a
$(v,k,1)$-BIBD (i.e., a \emph{balanced incomplete block design}).

For any fixed integer $t$, the set $U_t = \{ u : \text{ there exists a Room frame of type } t^u\}$ is PBD-closed. 
Thus, the natural approach is to construct a sufficient number of ``small'' Room frames by various appropriate techniques and then appeal to PBD-closure. For PBDs with block sizes not less than five, the following classical result due to Hanani is relevant.

\begin{theorem}\cite[Lemma 5.18]{Han75}
Let 
\[K_{\geq 5}  = \{ 5, 6, \dots, 20, 22,23,24,27,28,29,32,33,34,39\}.\]
Then, for all $u \geq 5$, there is a $(u,K_{\geq 5})$-PBD.
\end{theorem}

Therefore, if we can construct Room frames of type $2^u$ for all $u \in K_{\geq 5}$, we could then conclude that there is a Room frame of type $2^u$ for all $u \geq 5$.

We took the following approach in \cite{DS81}. First, we have already exhibited a Room frame of type $2^5$ in Figure \ref{type2^6}. For odd values $u > 5 \in K_{\geq 5}$, we made use a
``doubling construction'' which creates a Room frame of type $2^u$ from a skew Room square of side $u$. 
The idea is as follows:
\begin{construction}[Doubling Construction] 
{\rm \quad \\ \vspace{-.2in}
\begin{enumerate}
\item We construct a Room frame, say $F$, of type $1^u$ from a skew Room frame of side $u$. Suppose the holes are 
$\{i\}$, for $0 \leq i \leq u-1$. %\vspace{-.1in}
\item Construct a second  Room frame of type $1^u$ by transposing  $F$, in which the holes are $\{i'\}$, for $0 \leq i \leq u-1$. Superimpose $F$ and $F'$.
\item Construct a pair of latin squares, say $L$ and $L'$, of order $u$, having a common transversal. Assume that $L$ is on symbols $\{i'\}$, for $0 \leq i \leq u-1$, and $L'$ is on symbols $\{i'\}$, for $0 \leq i \leq u-1$. Superimpose $L$ and $L'$.  Assume the common transversal is $(i,i')$ for $0 \leq i \leq u-1$.
\item Construct an array of side $2u$ in which the superposition of $F$ and $F'$ is in the top left corner and the superposition of $L$ and $L'$ is in the bottom right corner. 
Finally, remove the common transversal from $L$ and $L'$.
The result is the desired Room frame of type $2^u$.
\end{enumerate}
}
\end{construction}

%For an illustration of the doubling construction, see example \ref{doubling.exam}.

\begin{example}
\label{doubling.exam}
We construct a Room frame of type $2^7$ from a skew Room square of side seven in Figure \ref{type2^7}. One of the holes is indicated in grey.
\begin{figure}[t]

\begin{center}

\begin{tikzpicture}[scale=0.18]

\filldraw [ultra thick, fill = gray!40] (0,70) -- (5,70) -- (5,65) -- (0,65) -- (0,70);
\draw [ultra thick, fill = gray!40] (35,70) -- (40,70) -- (40,65) -- (35,65) -- (35,70);
\draw [ultra thick, fill = gray!40] (0,35) -- (5,35) -- (5,30) -- (0,30) -- (0,35);
\draw [ultra thick, fill = gray!40] (35,35) -- (40,35) -- (40,30) -- (35,30) -- (35,35);

\draw [ultra thick] (0,0) -- (0,70);
\draw [thick] (5,0) -- (5,70);
\draw [thick] (10,0) -- (10,70);
\draw [thick] (15,0) -- (15,70);
\draw [thick] (20,0) -- (20,70);
\draw [thick] (25,0) -- (25,70);
\draw [thick] (30,0) -- (30,70);
\draw [ultra thick] (35,0) -- (35,70);
\draw [thick] (40,0) -- (40,70);
\draw [thick] (45,0) -- (45,70);
\draw [thick] (50,0) -- (50,70);
\draw [thick] (55,0) -- (55,70);
\draw [thick] (60,0) -- (60,70);
\draw [thick] (65,0) -- (65,70);
\draw [ultra thick] (70,0) -- (70,70);

\draw [ultra thick] (0,0) -- (70,0);
\draw [thick] (0,5) -- (70,5);
\draw [thick] (0,10) -- (70,10);
\draw [thick] (0,15) -- (70,15);
\draw [thick] (0,20) -- (70,20);
\draw [thick] (0,25) -- (70,25);
\draw [thick] (0,30) -- (70,30);
\draw [ultra thick] (0,35) -- (70,35);
\draw [thick] (0,40) -- (70,40);
\draw [thick] (0,45) -- (70,45);
\draw [thick] (0,50) -- (70,50);
\draw [thick] (0,55) -- (70,55);
\draw [thick] (0,60) -- (70,60);
\draw [thick] (0,65) -- (70,65);
\draw [ultra thick] (0,70) -- (70,70);

\node at (7.5,67) {$6,2$};
\node at (12.5,67) {$4,5$};
\node at (17.5,67) {$4',6'$};
\node at (22.5,67) {$1,3$};
\node at (27.5,67) {$2',3'$};
\node at (32.5,67) {$5',1'$};

\node at (2.5,62) {$6',2'$};
\node at (12.5,62) {$0,3$};
\node at (17.5,62) {$5,6$};
\node at (22.5,62) {$5',0'$};
\node at (27.5,62) {$2,4$};
\node at (32.5,62) {$3',4'$};

\node at (2.5,57) {$4',5'$};
\node at (7.5,57) {$0',3'$};
\node at (17.5,57) {$1,4$};
\node at (22.5,57) {$6,0$};
\node at (27.5,57) {$6',1'$};
\node at (32.5,57) {$3,5$};

\node at (2.5,52) {$4,6$};
\node at (7.5,52) {$5',6'$};
\node at (12.5,52) {$1',4'$};
\node at (22.5,52) {$2,5$};
\node at (27.5,52) {$0,1$};
\node at (32.5,52) {$0',2'$};

\node at (2.5,47) {$1',3'$};
\node at (7.5,47) {$5,0$};
\node at (12.5,47) {$6',0'$};
\node at (17.5,47) {$2',5'$};
\node at (27.5,47) {$3,6$};
\node at (32.5,47) {$1,2$};

\node at (2.5,42) {$2,3$};
\node at (7.5,42) {$2',4'$};
\node at (12.5,42) {$6,1$};
\node at (17.5,42) {$0',1'$};
\node at (22.5,42) {$3',6'$};
\node at (32.5,42) {$4,0$};

\node at (2.5,37) {$5,1$};
\node at (7.5,37) {$3,4$};
\node at (12.5,37) {$3',5'$};
\node at (17.5,37) {$0,2$};
\node at (22.5,37) {$1',2'$};
\node at (27.5,37) {$4',0'$};

\node at (42.5,32) {$6,2'$};
\node at (47.5,32) {$5,4'$};
\node at (52.5,32) {$4,6'$};
\node at (57.5,32) {$3,1'$};
\node at (62.5,32) {$2,3'$};
\node at (67.5,32) {$1,5'$};

\node at (47.5,27) {$0,3'$};
\node at (52.5,27) {$6,5'$};
\node at (57.5,27) {$5,0'$};
\node at (62.5,27) {$4,2'$};
\node at (67.5,27) {$3,4'$};
\node at (37.5,27) {$2,6'$};

\node at (52.5,22) {$1,4'$};
\node at (57.5,22) {$0,6'$};
\node at (62.5,22) {$6,1'$};
\node at (67.5,22) {$5,3'$};
\node at (37.5,22) {$4,5'$};
\node at (42.5,22) {$3,0'$};

\node at (57.5,17) {$2,5'$};
\node at (62.5,17) {$1,0'$};
\node at (67.5,17) {$0,2'$};
\node at (37.5,17) {$6,4'$};
\node at (42.5,17) {$5,6'$};
\node at (47.5,17) {$4,1'$};

\node at (62.5,12) {$3,6'$};
\node at (67.5,12) {$2,1'$};
\node at (37.5,12) {$1,3'$};
\node at (42.5,12) {$0,5'$};
\node at (47.5,12) {$6,0'$};
\node at (52.5,12) {$5,2'$};

\node at (67.5,7) {$4,0'$};
\node at (37.5,7) {$3,2'$};
\node at (42.5,7) {$2,4'$};
\node at (47.5,7) {$1,6'$};
\node at (52.5,7) {$0,1'$};
\node at (57.5,7) {$6,3'$};

\node at (37.5,2) {$5,1'$};
\node at (42.5,2) {$4,3'$};
\node at (47.5,2) {$3,5'$};
\node at (52.5,2) {$2,0'$};
\node at (57.5,2) {$1,2'$};
\node at (62.5,2) {$0,4'$};

\end{tikzpicture}

\end{center}
\caption{A Room frame of type $2^7$}
\label{type2^7}
\end{figure}

\end{example}

For values $u \in K_{\geq 5}$ that are divisible by four, we constructed orthogonal frame starters in $\zed_{2u} \setminus \{0,u\}$ to obtain the relevant Room frames. The remaining values $u \in K_{\geq 5}$, for which $u \equiv 2 \bmod 4$, were handled by orthogonal intransitive frame starters.

\subsubsection{The Case $t=4 $}% \quad  \vspace{-.1in}\\

The next cases to consider are for $t = 4$. This is similar to the $t=2$ case but a bit easier, because a Room frame of type $4^4$ exists while a Room frame of type $2^4$ does not exist. It is possible to use another PBD result due to Hanani:

\begin{theorem}\cite[Lemma 5.10]{Han75}
Let 
\[K_{\geq 4}  = \{ 4,7, \dots, 12, 14, 15,18,19,23,27\}.\]
Then, for all $u \geq 4$, there is a $(u,K_{\geq 4})$-PBD.
\end{theorem}

It is required to find a small number of Room frames of type $4^u$. The following standard result will be useful in the subsequent discussion.

\begin{theorem}[Inflation by MOLS]
\label{inflation.thm}
If there exists a Room frame of type $t^u$ and $s \neq 2,6$ is a positive integer, then there is a Room frame of type $(st)^u$.
\end{theorem}

The idea is to take $s$ copies of each point and replace each filled cell of the Room frame of type by an appropriate pair of orthogonal latin squares of side $s$. 

Let us return to our analysis of Room frames of type $4^u$. For $u$ odd, $u \geq 7$, we can start with a Room square of side $u$ (i.e., a Room frame of type $1^u$) and apply Theorem \ref{inflation.thm} (inflation by MOLS) with $s=4$. 
For $u=4$, see Example \ref{z4z4.exam}.  The case $u=5$ is a special case of a finite field starter-based construction from \cite{DS80}. For $u = 8,10,12,14$ and $18$, strong frame starters in cyclic groups yield the desired Room frames (see \cite{DS81}). The last case is $u=6$, which was handled in \cite{DS81} using orthogonal intransitve frame starters.

\subsubsection{General Values of $t$} %\quad  \vspace{-.1in}\\

Given the existence results for $t = 1$, $2$ and $4$ that are discussed above, we can handle most other values of $t$ by using Theorem \ref{inflation.thm} (inflation by MOLS). Starting with Room frames of type $2^u$ (for $u \geq 5$) and $4^u$ (for $u \geq 5$), we immediately get Room frames of type $t^u$ for all even $t$ and all $u \geq 5$. Similarly, starting with Room frames of type $1^u$  (for $u$ odd, $u \geq 7$), we obtain Room frames of type $t^u$ for all odd $t$ and all odd $u$, $u \geq 7$.

The remaining cases are Room frames of type $t^4$ (for even $t$) and $t^5$ (for odd $t$).  
As of 1981, Room frames of types $4^4$ and $8^4$ had been constructed using orthogonal frame starters. In conjunction with Theorem \ref{inflation.thm}, this showed that Room frames of type $t^4$ exist for all $t \equiv 0 \bmod 4$ (see \cite{DS81}). 

Also, by 1981, Room frames of types $2^5$, $3^5$, $5^5$ and $7^5$ had been constructed using orthogonal frame starters. Using Theorem \ref{inflation.thm}, this showed that Room frames of type $t^5$ exist  for all $t$ such that $\gcd(t, 210) > 1$ (see \cite{DS81}).

It was several years until these results were improved.
However, by the early 1990's, Room frames of types $6^4$ and $10^4$ were constructed  (see \cite{DS93}) using the hill-climbing algorithm described in \cite{DS87}, suitably modified to construct Room frames. Using Theorem \ref{inflation.thm}, it followed that Room frames of type $t^4$ exist for all $t$ divisible by $4$, $6$ or $10$. So all the remaining unknown cases for type $t^4$ had $t \equiv 2,10 \bmod 12$. 

A major advance was due to Ge and Zhu \cite{GZ93} in 1993, who utilized a sophisticated construction based on \emph{incomplete Room frames}. Their paper \cite{GZ93} solved all but a few cases of types $t^4$  and $t^5$, as stated in the following theorem.

\begin{theorem}
\cite{GZ93}
There exists a Room frame of type $t^4$ for all even $t \geq 4$, except possibly for 
$ t \in \{ 14,22,26,34,38,46,62,74,82,86,98,122,134,146\}$. Also, there exists a Room frame of type $t^5$ for all  $t \geq 5$, except possibly for $ t = 11, 13, 17$ or $19$.
\end{theorem}

The four possible exceptions of type $t^5$ were handled in a 1993 paper by Dinitz and Lamken \cite{DL93}, who constructed these Room frames using orthogonal frame starters. Then, in 1994, all but one of the possible exceptions of type $t^4$ were constructed by Dinitz, Stinson and Zhu \cite{DSZ94}. This was accomplished by a recursive construction that made use of starters having \emph{complete ordered transversals}. The single remaining exception was a Room frame of type $14^4$, which was finally constructed by Dinitz and Warrington \cite{DW10} in 2010. This ``last'' Room frame was constructed using the hill-climbing algorithm from \cite{DS92a}; it required over 5,000,000 trials before it completed successfully. 

An observant reader will note that Room frames of types $6^4$, $10^4$ and $14^4$ were constructed using hill-climbing. This is because is impossible to construct these Room frames from orthogonal frame starters 
(see \cite{St22}).

%\subsection{Applications}

%blah blah blah

\subsection{Some Historical Remarks}

The first Room frame can be found in the 1974 paper published by Wallis \cite{Wa74}. This was a Room frame  of type $2^5$, which I presented in Figure \ref{type2^5}.
Other early examples of Room frames in the literature include  a  Room frame of type $8^4$ (Wallis, \cite{Wa76}), and a skew Room frame of type $2^5$ (Beaman and Wallis, \cite{BW76}). These papers appeared in 1976 and 1977, respectively.

%Today we would call it a Room frame of type $2^5$. However, 
These papers from the 1970's did not define Room frames  as a specific combinatorial structure. 
For example, the ``original'' Room frame, of type $2^5$, was simply a component used in a construction that created a Room square of side $5(v-w)+w$ from a Room square of side $v$ containing a Room subsquare of side $w$. This construction was used to obtain a Room square of side $257$, which completed the solution of the Room square existence problem (we provide a few more details below). %However, it turned out that Room frames, and various generalizations (especially, Kirkman frames) were an amazingly powerful tool in the investigation of many combinatorial design problems. %This paper is primarily a discussion of the history of frames and their applications.

The term ``frame'' was apparently first coined in the 1981 paper by Mullin, Schellenberg, Vanstone and Wallis \cite{MSVW81}. This paper refers to the 1972 survey by Wallis \cite{Wa72} as the place where Room frames were introduced. However, although \cite[Chapter IV]{Wa72} discusses a construction that is termed ``the frame construction,'' there does not seem to be any actual use of Room frames (as we now understand the term) there. 

It is  clear that the paper   \cite{MSVW81} was the first one where Room frames were studied as objects of interest in their own right. (However, I should point out that the Room frames studied in  \cite{MSVW81} were in fact skew Room frames.) The main result proven in  \cite{MSVW81} was that a skew Room frame of type $2^n$ exists for all positive integers $n \equiv 1 \bmod 4$, $n \neq 33,57,93,133$. Skew Room frames of types $2^n$ were known to exist for $n = 5,9,13,17$ and the set $\{ n : \text{there exists a skew Room frame of type $2^n$}\}$ is PBD-closed. Constructions of PBDs with blocks of sizes $5,9,13$ and $17$ given in \cite{MSVW81} completed the proof.

I became aware of the paper \cite{MSVW81} in 1978 when Ron Mullin gave me a preprint version of the paper (the was no arXiv in those days!). I found the idea of Room frames quite fascinating and they became an important technique in my research tool chest for many years. Room frames were in fact a central theme in my PhD thesis \cite{St-Phd} and I subsequently published a number of papers focussed on Room frames and their applications starting in the early 1980's. %Two papers \cite{DS80,DS81} were co-authored with Jeff Dinitz, one paper \cite{SW80} was co-authored with Wal Wallis, and I wrote a few ``frame'' papers on my own, including \cite{St81,St81a,St81b}. %I also was involved in a number of other papers in which frames were used to construct other types of designs. These will be discussed in subsequent sections of this paper.

%\subsection{Uses of Room Frames}

%In this section, we briefly recall some applications of Room frames in the construction of other designs.

%\subsection{Direct Product Constructions for Room Squares}
%\label{dp.sec}

I would like to comment briefly on direct product constructions, which have had a long history in combinatorial designs. A singular direct product construction was in fact the original motivation for Room frames (see Wallis \cite{Wa74}). In the case of Room squares, a \emph{direct product} construction creates a Room square of side $uv$ from Room squares of sides $u$ and $v$. A \emph{singular direct product} creates a Room square of side $u(v-w) + w$ from a Room square of side $u$ and a Room square of side $v$ that contains a Room subsquare of side $w > 0$.
The singular direct product requires that there exists a pair of orthogonal latin squares of order $v-w$, which rules out the ordered pair $(v,w) = (7,1)$, since a pair of orthogonal latin squares of order six does not exist. The resulting Room square of side $u(v-w) + w$  contains Room subsquares of sides $u, v$ and $w$. 

Since a Room square of side five does not exist, we cannot take $u=5$ in the singular direct product.
However, the existence of a Room frame of type $2^5$ provides a clever way to circumvent this restriction. 
Given a Room square of side $v$ that contains a Room subsquare of side $w > 0$, a singular direct product uses the Room frame of type $2^5$ to create a Room square of side $5(v-w)+w$, provided that $v - w \neq 12$. (The reason for the restriction $v - w \neq 12$ is that this variation of the singular direct product requires a pair of orthogonal latin squares of order $(v-w)/2$.)

The previously mentioned construction of a Room square of side 257 (from \cite{Wa74}) uses two applications of the singular direct product:
\begin{eqnarray*}
57 & = & 7(9-1) + 1\\
257 &=& 5(57-7)+7.
\end{eqnarray*}
The first equation leads to a Room square of side 57 that contains a Room subsquare of side seven. 
This is then used, in the ``Room frame'' variation of the singular direct product, in the second equation.

\section{Kirkman Frames}

\subsection{Motivation: From Room frames to Kirkman Frames}

A couple of years after getting my PhD---probably around 1983---I started thinking about generalizations of Room frames. The obvious place to begin was to look at block size three rather than block size two. At the same time, it seemed simplest to start with a single resolution rather than the orthogonal resolutions that exist in Room squares and Room frames. So this led to the definition of a Kirkman frame as a ``Kirkman triple system with holes'' that I gave in my paper \cite{St87}, which was published in 1987.

There was a long delay in the publication of this paper, as it was rejected by at least three journals before it was accepted by \emph{Discrete Mathematics}. Ironically, it has turned out be one of my most highly cited papers in design theory. The following quote is from the 2003 survey by Rees and Wallis \cite[p.\ 332]{RW03}:
\begin{quote}
\emph{``Since their introduction by Stinson, however, Kirkman Frames
have proven to be the single most valuable tool for the construction of
the various generalizations of KTSs that will be discussed in this survey.''}
\end{quote}
I should also mention the book by Furino, Miao and Yin \cite{FMY96}, which is devoted to the topic of Kirkman frames, as evidence of the importance of this topic.

In a Room frame, a hole of size $t$ intersects $t$ rows and $t$ columns of the array. This is often stated as part of the definition. In the case of a Kirkman frame, I decided to simply require that the set of blocks could be partitioned into \emph{holey parallel classes}, where each holey parallel class forms a partition of all the points not in some hole. It then can be proven using a simple counting argument that a hole of size $t$ is associated with exactly $t/2$ holey parallel classes. This of course implies that every hole has even size. 

It is easy to see that a Kirkman triple system of order $v$ is equivalent to a Kirkman frame of type $2^{v/2}$. %So the next hole sizes to consider are holes of size four and size six. Most of these can be constructed recursively from Kirkman frame of type $2^n$. 
So the smallest Kirkman frame that is not equivalent to a Kirkman triple system is the Kirkman frame of type $4^4$. Example \ref{type4^4} provides a construction of this Kirkman frame.

\begin{example}
\label{type4^4}
{\rm I construct a Kirkman frame of type $4^4$ using the technique that I described in \cite{St86}. Here is a pair of incomplete orthogonal latin squares of order six with a hole of size two, which were discovered by Euler in the eighteenth century:

\medskip

\[
\begin{array}{|x{.3cm}|x{.3cm}|x{.3cm}|x{.3cm}|x{.3cm}|x{.3cm}|}
\hline 
$5$ & $6$ & $3$ & $4$ & $1$ & $2$\\ \hline 
$2$ & $1$ & $6$ & $5$ & $3$ & $4$\\ \hline 
$6$ & $5$ & $1$ & $2$ & $4$ & $3$\\ \hline 
$4$ & $3$ & $5$ & $6$ & $2$ & $1$\\ \hline 
$1$ & $4$ & $2$ & $3$ &   &  \\ \hline 
$3$ & $2$ & $4$ & $1$ &   &  \\ \hline 
\end{array}
\hspace{1in}
\begin{array}{|x{.3cm}|x{.3cm}|x{.3cm}|x{.3cm}|x{.3cm}|x{.3cm}|}
\hline 
$a$ & $b$ & $e$ & $f$ & $c$ & $d$\\ \hline 
$f$ & $e$ & $a$ & $b$ & $d$ & $c$\\ \hline 
$d$ & $c$ & $f$ & $e$ & $a$ & $b$\\ \hline 
$e$ & $f$ & $d$ & $c$ & $b$ & $a$\\ \hline 
$b$ & $d$ & $c$ & $a$ &   &  \\ \hline 
$c$ & $a$ & $b$ & $d$ &   &  \\ \hline 
\end{array}
\]

\medskip
 
 From these incomplete orthogonal latin squares, construct an incomplete group-divisible design. I label the rows and columns and obtain a block of size four from each of the $32$ filled cells:
 \[
 \begin{array}{llll}
 \{r_1,c_1,5,a\} & \{r_1,c_2,6,b\} &  \{r_1,c_3,3,e\} & \{r_1,c_4,4,f\} 
\\  \{r_1,c_5,1,c\} & \{r_1,c_6,2,d\}\\
 \{r_2,c_1,2,f\} & \{r_2,c_2,1,e\} &  \{r_2,c_3,6,a\} & \{r_2,c_4,5,b\} 
\\  \{r_2,c_5,3,d\} & \{r_2,c_6,4,c\}\\
 \{r_3,c_1,6,d\} & \{r_3,c_2,5,c\} &  \{r_3,c_3,1,f\} & \{r_3,c_4,2,e\} 
\\  \{r_3,c_5,4,a\} & \{r_3,c_6,3,b\}\\
 \{r_4,c_1,4,e\} & \{r_4,c_2,3,f\} &  \{r_4,c_3,5,d\} & \{r_4,c_4,6,c\} 
\\  \{r_4,c_5,2,b\} & \{r_4,c_6,1,a\}\\
 \{r_5,c_1,1,b\} & \{r_5,c_2,4,d\} &  \{r_5,c_3,2,c\} & \{r_5,c_4,3,a\} \\
 \{r_6,c_1,3,c\} & \{r_6,c_2,2,a\} &  \{r_6,c_3,4,b\} & \{r_6,c_4,1,d\} 
\end{array}
 \]
 Every block contains exactly one point from $\{ 5,6,e,f,r_5,r_6,c_5,c_6\}$.
 Then delete these eight points, obtaining $32$ blocks of size three. Each deleted point gives rise to a holey parallel class. The resulting Kirkman frame has holes $\{1,2,3,4\}$, $\{a,b,c,d\}$, $\{ r_1,r_2,r_3,r_4\}$
 and $\{ c_1,c_2,c_3,c_4\}$. The blocks, arranged into eight holey parallel classes, are as follows:
 \[
 \begin{array}{l|l}
 \multicolumn{2}{c}{\{1,2,3,4\}}   \\ \hline
\{r_1,c_1,a\} & \{r_1,c_2,b\}   \\
\{r_2,c_4,b\} & \{r_2,c_3,a\}   \\
\{r_3,c_2,c\} & \{r_3,c_1,d\}    \\
\{r_4,c_3,d\} & \{r_4,c_4,c\}    \\
 \end{array}
 \hspace{1in}
\begin{array}{l|l}
 \multicolumn{2}{c}{\{a,b,c,d\}} \\ \hline
  \{r_1,c_3,3\} & \{r_1,c_4,4\}  \\
  \{r_2,c_2,1\} & \{r_2,c_1,2\}  \\
  \{r_3,c_4,2\} & \{r_3,c_3,1\}  \\
  \{r_4,c_1,4\} & \{r_4,c_2,3\}  \\
 \end{array}
\]
 
\[
 \begin{array}{l|l}
 \multicolumn{2}{c}{\{ r_1,r_2,r_3,r_4\}}   \\ \hline
\{c_1,1,b\} &  \{c_1,3,c\} \\
 \{c_2,4,d\} &  \{c_2,2,a\}  \\ 
  \{c_3,2,c\} &   \{c_3,4,b\} \\ 
\{c_4,3,a\}  &  \{c_4,1,d\} \\
 \end{array}
 \hspace{1in}
\begin{array}{l|l}
 \multicolumn{2}{c}{\{ c_1,c_2,c_3,c_4\}} \\ \hline
 \{r_1,1,c\} & \{r_1,2,d\}  \\
 \{r_2,3,d\} & \{r_2,4,c\}  \\ 
 \{r_3,4,a\} & \{r_3,3,b\}  \\ 
  \{r_4,2,b\} & \{r_4,1,a\} \\
   \end{array}
\]
}
\hfill$\blacksquare$
\end{example}

In  \cite{St87}, I proved that there is a Kirkman frame of type $t^u$ if and only if $t$ is even and $t(u-1) \equiv 0 \bmod 3$. I will review the main steps in the proof. The proof used PBDs and group-divisible designs (GDDs) along with a few small Kirkman frames. One useful recursive tool is the 
``GDD Construction'' from \cite{St87}. This is just a standard Wilson-type GDD construction (see, e.g., \cite{Wi75}), adapted to the setting of Kirkman frames. 

I recall a special case of the GDD Construction that is sufficient for our needs, but first I define 
group-divisible designs (GDDs). A \emph{group-divisible design} is a triple
$(X,\GG,\AAA)$, where $X$ is a finite set of \emph{points}, $\GG$ is a partition of $X$ into subsets called \emph{groups\footnote{Of course these are \emph{not} algebraic groups.}}, and $\AAA$ is a set of subsets of $X$, with the property that every pair of points is contained in a unique group, or a unique block, but not both. Assume that $|A| > 1$ for every $A \in \AAA$. 
Observe that a PBD is equivalent to a GDD in which every group has size $1$.

%\begin{theorem}
%[GDD construction]
%Let $(X, \GG , \AAA)$ be a GDD, and let $w:X \rightarrow \zed^+ \cup \{0\}$. (We say that $w$ is a
%\emph{weighting}). For each block $A \in \AAA$, suppose there is a frame of type ${w(x):x \in A}$. Then there is a frame of type $\{\sum_{x \in G} w(x): G \in \GG\}$.
%\end{theorem}

\begin{theorem}
[GDD construction]
\label{GDD.const}
Let $(X, \GG , \AAA)$ be a GDD in which $|X| = gu$ and there are $u$ groups of size $g$, and let $w\geq 1$ ($w$ is often called a \emph{weight}). For each block $A \in \AAA$, suppose there is a Kirkman frame of type $w^{|A|}$. Then there is a Kirkman frame of type $(gw)^u$.
\end{theorem}

\begin{remark}
If $|G| = 1$ for all $G \in \GG$ (i.e., the GDD is a PBD), then we obtain a PBD-closure result. More precisely, the set 
\[\{u : \text{there exists a Kirkman frame of type } t^u \}\] is PBD-closed for any fixed $t$.
\end{remark}

Now I discuss the existence proof for uniform Kirkman frames. When $t=2$, it follows that $u \equiv 1 \bmod 3$. Thus all the Kirkman frames in the case $t=2$ can immediately be obtained from Kirkman triple systems. 

The next case, $t=4$, is easily handled as follows. Again, $u \equiv 1 \bmod 3$ is a necessary condition. I  presented a Kirkman frame of type $4^4$ in Example \ref{type4^4}. When $u \equiv 1 \bmod 3$, $u \geq 7$, it suffices to apply Theorem \ref{GDD.const} as follows. Start with a group-divisible design having $u$ groups of size two and blocks of size four, and define $w=2$. Every block is replaced by a Kirkman frame of type $2^4$. The required GDDs are constructed in 
\cite{BHS77}.

The case $t=6$ is only a bit more difficult. Here there are no congruential conditions on $u$. I split the proof into two subcases, namely $u \equiv 0,1 \bmod 4$ and $u \equiv 2,3 \bmod 4$.  

When $u \equiv 0,1 \bmod 4$, start with a $(3u+1,\{4\})$-PBD (i.e., a $(3u+1,4,1)$-BIBD). Delete a point, creating a GDD with $u$ groups of size three and blocks of size four. Define $w=2$ and apply Theorem \ref{GDD.const}, filling in Kirkman frames of type $2^4$. The result is a Kirkman frame of type $6^u$.

For $u \equiv 2,3 \bmod 4$, $u \geq 7$, start with a $(3u+1,\{4,7\})$-PBD that contains a unique block of size seven (see \cite{Br79}). Delete a point that is not in the block of size seven, creating a GDD with $u$ groups of size three and blocks of size four and seven. Define $w=2$ and apply Theorem \ref{GDD.const}, filling in Kirkman frames of type $2^4$ and $2^7$. The result is a again Kirkman frame of type $6^u$.

It remains to construct a Kirkman frame of type $6^6$. A direct construction for this Kirkman frame can be found in \cite{RS92}.

Having handled the cases $t = 2,4$ and $6$, all other cases follow by ``Inflation by MOLS'' (Theorem \ref{inflation.thm}), which also works for Kirkman frames. The reader can fill in the details.

%\subsection{Applications}

%In \cite{St87}, I used Kirkman frames to  prove that there is a Kirkman triple system of order $v$ that contains a Kirkman triple system of order $w$ for all $v,w \equiv 3 \bmod 6$, $v \geq 4w - 9$. 

%\subsection{constructions for Kirkman Frames}

\subsection{Some Historical Remarks}
\label{hist.sec}

Kirkman frames can be used in existence proofs for Kirkman triple systems. In particular, several of the constructions for Kirkman triple systems given by Ray-Chaudhuri and Wilson in \cite{RW71}  (see also \cite{HRW72}) can be recast as Kirkman frame-based constructions. For example, \cite[Theorem 1]{RW71} proves that the set 
\[R_3^* = \{ r: \text{there exists a Kirkman triple system of order } 2r+1\}\] is PBD-closed.
That is, if there exists a $(v,K)$-PBD where $K \subseteq R_3^*$, then $v \in R_3^*$.\footnote{Note that this result is a special case of Theorem \ref{GDD.const}.}

I sketch the proof of this fundamental result using Kirkman frame language, making use of the fact that a Kirkman frame of type $2^r$ is equivalent to a Kirkman triple system of order $2r+1$. Basically, it suffices to delete a point from the Kirkman triple system to obtain the desired Kirkman frame (see Example 
\ref{KTS-frame.exam}). Then apply Theorem \ref{GDD.const} with $w=2$. (In more detail, let $(X, \BB)$ be a $(v,K)$-PBD where $K \subseteq R_3^*$. Give every point in $X$ weight two. For every block $B \in \BB$, construct a Kirkman frame of type $2^{|B|}$ on the points $B \times \{1,2\}$, where the holes are
$\{x\} \times \{1,2\}$, $x \in B$. The result is a Kirkman frame of type $2^v$, which is equivalent to a 
Kirkman triple system of order $2v+1$.)

Interestingly, the original proof of this result, given in \cite{RW71}, is slightly different. Rather than utilizing Kirkman frames (implicitly or explicitly), the proof instead uses the notion of a \emph{completed resolvable design}. The idea is to add a new point to the blocks of each parallel class of a given Kirkman triple system of order $2r+1$. Then a new block is formed, consisting of the $r$ new points. The result is a 
$(3r+1, \{4,r\})$-PBD (and this PBD contains a unique block of size $r$ if $r > 4$). Clearly this  process can be reversed; see Example \ref{KTS-frame.exam}. The PBD-closure proof given 
in \cite{RW71} starts with a PBD and gives every point weight 3. Every block is replaced by a suitable completed resolvable design. At the end, a completed resolvable design is constructed; this design is equivalent to a Kirkman triple system of order $2v+1$. 

The following example illustrates a KTS, a Kirkman frame and a completed resolvable design, all of which are equivalent structures.

\begin{example}
\label{KTS-frame.exam} 
{\rm I present a Kirkman triple system of order 15 that was originally discovered by Cayley in 1850 (see \cite[Example 19.7]{CR99}). This Kirkman triple system has the following 35 blocks, where each row of five blocks is a parallel class: 
\[
\begin{array}{l@{\hspace{.2in}}l@{\hspace{.2in}}l@{\hspace{.2in}}l@{\hspace{.2in}}l}
\mathit{abc} & \mathit{d}35 & \mathit{e}17 & \mathit{f}28 & \mathit{g}46\\
\mathit{ade} & \mathit{b}26 & \mathit{c}48 & \mathit{f}15 & \mathit{g}37\\
\mathit{afg} & \mathit{b}13 & \mathit{c}57 & \mathit{d}68 & \mathit{e}24\\
\mathit{bdf} & \mathit{a}47 & \mathit{c}16 & \mathit{e}38 & \mathit{g}25\\
\mathit{beg} & \mathit{a}58 & \mathit{c}23 & \mathit{d}14 & \mathit{f}67\\
\mathit{cdg} & \mathit{a}12 & \mathit{b}78 & \mathit{e}56 & \mathit{f}34\\
\mathit{cef} & \mathit{a}36 & \mathit{b}45 & \mathit{d}27 & \mathit{g}18
\end{array}
\]
To construct the associated Kirkman frame of type $2^7$, delete a point. The holes of the Kirkman frame are formed by the blocks containing the given point. Suppose the point $\mathit{a}$ is deleted. Then we obtain the following Kirkman frame:
\[
\begin{array}{c|l@{\hspace{.2in}}l@{\hspace{.2in}}l@{\hspace{.2in}}l}
\text{hole} & \multicolumn{4}{|c}{\text{holey parallel class}}\\ \hline 
\mathit{bc} & \mathit{d}35 & \mathit{e}17 & \mathit{f}28 & \mathit{g}46\\
\mathit{de} & \mathit{b}26 & \mathit{c}48 & \mathit{f}15 & \mathit{g}37\\
\mathit{fg} & \mathit{b}13 & \mathit{c}57 & \mathit{d}68 & \mathit{e}24\\
47 &\mathit{bdf} &  \mathit{c}16 & \mathit{e}38 & \mathit{g}25\\
58 & \mathit{beg} &  \mathit{c}23 & \mathit{d}14 & \mathit{f}67\\
12 & \mathit{cdg} &  \mathit{b}78 & \mathit{e}56 & \mathit{f}34\\
36 & \mathit{cef} &  \mathit{b}45 & \mathit{d}27 & \mathit{g}18
\end{array}
\]
Notice that there are seven holes of size two, and seven holey parallel classes.

Finally, I construct the associated completed resolvable design. For each of the seven parallel classes, add a new point to all the blocks in that parallel class. Then create a new block consisting of the seven new points.
\[
\begin{array}{l@{\hspace{.2in}}l@{\hspace{.2in}}l@{\hspace{.2in}}l@{\hspace{.2in}}l}
\infty_1\mathit{abc} & \infty_1\mathit{d}35 & \infty_1\mathit{e}17 & \infty_1\mathit{f}28 & \infty_1\mathit{g}46\\
\infty_2\mathit{ade} & \infty_2\mathit{b}26 & \infty_2\mathit{c}48 & \infty_2\mathit{f}15 & \infty_2\mathit{g}37\\
\infty_3\mathit{afg} & \infty_3\mathit{b}13 & \infty_3\mathit{c}57 & \infty_3\mathit{d}68 & \infty_3\mathit{e}24\\
\infty_4\mathit{bdf} & \infty_4\mathit{a}47 & \infty_4\mathit{c}16 & \infty_4\mathit{e}38 & \infty_4\mathit{g}25\\
\infty_5\mathit{beg} & \infty_5\mathit{a}58 & \infty_5\mathit{c}23 & \infty_5\mathit{d}14 & \infty_5\mathit{f}67\\
\infty_6\mathit{cdg} & \infty_6\mathit{a}12 & \infty_6\mathit{b}78 & \infty_6\mathit{e}56 & \infty_6\mathit{f}34\\
\infty_7\mathit{cef} & \infty_7\mathit{a}36 & \infty_7\mathit{b}45 & \infty_7\mathit{d}27 & \infty_7\mathit{g}18\\
\multicolumn{5}{l}{\infty_1\infty_2\infty_3\infty_4\infty_5\infty_6\infty_7} 
\end{array}
\]
}
\hfill$\blacksquare$
\end{example}

On the other hand, as Assaf and Hartman pointed out in \cite{AH89}, Hanani constructs frames with block size three and $\lambda = 2$ in his 1974 paper \cite{Han74}\footnote{We discuss frames with $\lambda > 1$ in greater detail in Section \ref{other.sec}.}. These frames are the same thing as \emph{near-resolvable $(v,3,2)$-BIBDs}. There are $v$ holey parallel classes in such a BIBD, each of which misses a single point (thus a necessary condition for existence is that $v \equiv 1 \bmod 3$). Hanani gives direct constructions for several such frames, and he also proves a PBD-closure result in \cite[Lemma 8]{Han74}. This establishes the existence of frames of type $1^{3u+1}$ (with block size three and $\lambda = 2$) for all positive integers $u$. 

Assaf and Hartman \cite{AH89} also note that the paper \cite{Han74}  constructs examples of frames with 
holes of size $t$, for $t =3,12,24$ (again, these frames have block size three and $\lambda = 2$).

It is somewhat difficult to identify the first paper to actually include a Kirkman frame construction (i.e., one with $\lambda = 1$  that incorporates a resolution into holey parallel classes).
Perhaps it is the 1977 Baker-Wilson paper  (\cite{BW77}) on nearly Kirkman triple systems (NKTS). 
Here is the statement and proof of a main recursive construction, exactly as it appears in \cite{BW77}.

\begin{theorem}
\cite[Lemma 3]{BW77}
If there exists a GDD on $u$ points with group sizes from $M = \{g_1, \dots , g_m\}$ and block sizes from 
$K = \{k_1, \dots , k_{\ell}\}$ such that: (i) for each $g_i$ there exists an NKTS$[2g_i+2]$, and (ii) for each $k_i$ there exists a KTS$[2k_i+1]$, then there exists an NKTS$[2u+2]$. 
\end{theorem}

\begin{proof}
Let the GDD have points $Y$ and sets $\GG \cup \BB$. The point set $X$ of the NKTS$[2u+2]$ is to consist of
$\theta'$, $\theta''$ and two symbols $y'$, $y''$ for each $y \in Y$. The matching is to be $\AAA_0 = \{ \{ \theta', \theta''\}\} \cup \{\{y',y''\}: y \in Y\}$, and the parallel classes are to be $\AAA_y$, $y \in Y$, obtained as follows:
For each $G \in \GG$, form an NKTS with matching $\{ \{ \theta', \theta''\}\} \cup \{\{y',y''\}: y \in G\}$
and parallel classes $\AAA_y^G$, arbitrarily indexed by the elements $y \in G$. For each $B \in \BB$, form a KTS on the points $y'$, $y''$, $y \in B$ and another symbol $*$ so that $\{*, y', y''\}$ is a triple of the KTS for each $y \in B$; the other triples which belong to the parallel class containing $\{*, y', y''\}$ will be denoted by $\AAA_y^B$ (these triples partition the set of $z'$, $z''$ with $z \in B$, $z \neq y$). Then the parallel class $\AAA_y$ is to be the union of $\AAA_y^G$, where $G$ is the unique member of $\GG$ containing $y$, and all the classes $\AAA_y^B$, where $B$ is a member of $\BB$ containing $y$.
\end{proof}

Now where is the Kirkman frame in the above proof?  Let $B \in \BB$ be any block. For each $y \in B$, we observe that each $\AAA_y^B$ is a holey parallel class with hole $\{y', y''\}$. Further, the union of these holey parallel classes forms a Kirkman frame of type $2^{|B|}$. We can describe this construction informally, omitting details and using Kirkman frame-based language, as follows:
\begin{enumerate}
\item start with a GDD, take two copies of every point and take two new points
\item replace every block $B$ by a Kirkman frame of type $2^{|B|}$, and 
\item replace every group $G$ by an NKTS$[2|G| + 2]$ containing the two new points.
\end{enumerate} 

Finally, I should  mention  the paper by Lee and Furino (\cite{LF95}) entitled 
\emph{A translation of J. X. Lu's ``an existence theory for resolvable balanced incomplete block designs''}. Lu's 1984 paper is another early example of a paper that uses Kirkman frame techniques in the context of resolvable designs.
According to Furino and Lee,
\begin{quote}
\emph{``We have remained faithful to the constructions provided by Lu but have altered the presentation to be more consistent with contemporary techniques and notation. Specifically, Lu’s Theorem 4 is presented via frames which are never explicitly mentioned in the original but which are clearly embedded in the constructions.''
}
\end{quote}

\section{Other Kinds of Frames}
\label{other.sec}

Various generalizations of frames have been studied. We briefly discuss some possible aspects that have been considered.

\subsection{Frames with Larger Block Size} 

Room frames have \emph{block size} two and Kirkman frames have block size three. Frames with larger block size can also be studied.
The most obvious generalization of Kirkman frames would be frames with block size four, which are often known as \emph{$4$-frames}
(analogously, a frame with block size $k$ is termed a \emph{$k$-frame}).
These were first studied systematically by Rees and Stinson in \cite{RS92}, and additional results can be found in \cite{CSZ97,WG14,ZG07}. The following theorem summarizes the current existence results for $4$-frames.

\begin{theorem}
There exists a $4$-frame of type $h^u$ if and only if $u\geq 5$, $h \equiv 0 \pmod{3}$ and 
$h(u - 1)  \equiv 0 \pmod{4}$, except
possibly where
\begin{enumerate}
\item  $h = 36$ and $u = 12$;
\item  $h \equiv 6 \pmod{12}$ and
\begin{enumerate}
\item $h = 6$ and $u \in \{7, 23, 27, 35, 39, 47\}$;
\item $h = 18$ and $u \in \{15, 23, 27\}$;
\item $h \in \{30, 66, 78, 114, 150, 174, 222, 246, 258, 282, 318, 330, 354, 534\}$ and\\ 
$u \in \{7, 23, 27, 39, 47\}$;
\item $h \in \{n : 42 \leq n \leq 11238\} \setminus \{66, 78, 114, 150, 174, 222, 246, 258, 282, 318, 330, 354, 534\}$ and $u \in \{23, 27\}$.
\end{enumerate}
\end{enumerate}
\end{theorem}

We should also mention that it is easy to see that a resolvable $(v,k,1)$-BIBD is equivalent to a $k$-frame of type $(k-1)^r$, where $r = (v-1)/(k-1)$. It suffices to delete a point from the BIBD to construct the frame, and the process can be reversed.

\subsection{Frames with Larger Dimension} 

Room frames have \emph{dimension} $two$ (since they have two orthogonal frame resolutions) and, analogously, Kirkman frames have dimension one.  Frames of higher dimension have also received attention.  For example, it was shown in \cite{DS80} that a Room frame of type $2^q$ having  dimension $t$ can be constructed 
if $q = 2^k t + 1$ is a prime power, $t \geq 3$ is odd and $k \geq 2$. 
There have also been numerous papers addressing the case of two-dimensional Kirkman frames; see \cite{DAW15} for recent results.

\subsection{Frames for Graph Decompositions}

The ``blocks'' in Room and Kirkman frames correspond to complete graphs ($K_2$ and $K_3$, resp.) \emph{Graph decompositions} into graphs other than complete graphs are also possible in a frame setting.

\emph{Cycle frames} were introduced in \cite{ASSW89}. Here we have a decomposition of a complete multipartite graph into holey parallel classes of cycles of lengths three and five. These cycle frames were useful in solving the uniform Oberwolfach problem for odd length cycles. There has also been considerable research done on cycle frames where there is only one cycle length; see \cite{BCDT17} for a very general existence result.
 
Another variation is  \emph{$K_{1,3}$-frames} \cite{CC17}. This is a frame-type graph decomposition into graphs isomorphic to  $K_{1,3}$. And, as might be expected, a  \emph{$(K_{4} - e)$-frame} (see \cite{CSZ97}) involves decompositions into graphs isomorphic to $K_4$ with an edge deleted.
 
\subsection{Frames with $\lambda > 1$}

Various types of frames can also be considered for $\lambda > 1$, i.e., where every pair in the underlying design occurs $\lambda > 1$ times. For example, two-dimensional Room frames having block size three and $\lambda = 2$ were first studied in \cite{La91,LV93}.

In general, we can consider frames with block size $k$ and a specified value of $\lambda$; such a frame is commonly denoted as a $(k,\lambda)$-frame.   For example, necessary and sufficient existence conditions for $(3,\lambda)$-frames are determined in \cite{AH89}. There also has been considerable work done on $(4,3)$-frames; see, e.g., \cite{FKLMY,Ge01,GLL}.

Finally, a near-resolvable $(v,k,k-1)$-BIBD is the same thing as a $(k,k-1)$-frame of type $1^v$, since there are $v$ holey parallel classes in the BIBD, each missing one point
(we already mentioned this result in Section \ref{hist.sec} in the case $k=3$).

%%%%%%%%%%

\subsection{Nonuniform Frames}

A \emph{nonuniform} frame is one in which not all the holes have the same size. Nonuniform frames are often used in constructions of uniform frames. Also, Room squares or Kirkman triple systems containing subdesigns are equivalent to certain frames where all but one of the holes have the same size. For example, a Room square of side $v$ containing a Room subsquare of side $w$ is equivalent to a Room frame of type $1^{v-w}w^1$. Additionally, various other classes of nonuniform frames have been studied, e.g., Room frames of type
$2^tu^1$ (\cite{DSZ94}).

\subsection{Frames with Special Properties}

There has been some study of frames with various special properties. We have already mentioned skew Room frames in Section \ref{intro.sec}.
Other examples of Room frames having special properties that have been studied include \emph{partitionable skew Room frames} \cite{CSZ97,ZG07}, as well as
Room frames with \emph{partitionable transversals} \cite{DL00}.

Finally, I should mention a generalization of frames known as \emph{double frames}. These objects were defined by Chang and Miao in \cite{CM02}, where they were used to unify various frame-type recursive constructions. One application of double frames is to construct frames, e.g., see \cite{WG14}.

\subsection{Equiangular Tight Frames}

It is  not surprising that a term such as ``frame'' could have multiple meanings, even within combinatorial mathematics. However, it is probably not to be expected that the phrase ``Kirkman frame'' would arise in two completely different contexts. 
To be specific, ``Kirkman equiangular tight frames and codes'' is the title of a 2014 paper by Jasper, Mixon and Fickus 
 \cite{JMF14}.
 
In this context, the term ``frame'' can be found in the 1992 book by Daubechies \cite{Da92}; it refers to a generalization of an orthonormal basis. The \emph{equiangular tight frames} (see \cite{StHe}) meet the Welch bound with equality. Many of the known constructions for these objects use combinatorial designs such as conference matrices, difference sets and Steiner systems.
In \cite{JMF14}, a construction utilizing Kirkman triple systems was proposed; the resulting frames were termed ``Kirkman equiangular tight frames.'' Thus we have two completely different notions of a Kirkman frame!

\section{Applications}

In this paper, I have concentrated on construction methods for Room frames and Kirkman frames. However, I would be remiss if I did not mention some applications of frames. Here are a few ``obvious'' applications. Room frames were essential in proving the existence of Room squares with subsquares (see \cite{DSZ94,DW10}), and analogously, Kirkman frames were of crucial importance in constructing Kirkman triple systems with Kirkman subsystems (\cite{RS89}). Skew Room frames were a very important tool in solving the skew Room square existence problem (\cite{St81c}), and Kirkman frames were employed in the construction of resolvable group-divisible designs with block size three (\cite{RS86}). The survey by Rees and Wallis \cite{RW03} is especially useful as it gives detailed discussion and self-contained proofs of several important applications of Kirkman frames.
Finally, frames with larger block size (in particular, $4$-frames) are utilized in similar problems relating to designs with larger block size (see, e.g., \cite{WG14}).

There are many other applications that are perhaps not immediately obvious. I will list a few representative examples now; however, I should note that this is far from being a complete list.
\begin{itemize}
\item Skew Room frames were used in \cite{LRS} and elsewhere to construct nested cycle systems.
\item Skew Room frames were used in \cite{Ro94,RWCZ} to construct weakly $3$-chromatic BIBDs with block size four.
\item Partitionable skew Room frames can be used to construct resolvable $(K_{4} - e)$-designs (see \cite{CSZ97}). For  information on  partitionable skew Room frames, see \cite{ZG07}.
\item $4$-frames have been used to construct three mutually orthogonal latin squares with holes (i.e., three HMOLS); see \cite{SZ91,CSZ97}.
\item Uniformly resolvable designs, which were introduced by Rees \cite{Re87}, have been studied by numerous authors.
Their construction often employs frames. Two recent papers on this topic are \cite{WG16,WG17}.
\item Ling \cite{Li03} uses splittable $4$-frames to give improved results concerning a problem in generalized Ramsey theory involving edge-colourings of $K_{n,n}$.
\end{itemize}

\end{document}